\documentclass[10pt]{amsart}
\usepackage{geometry}                
\usepackage{graphicx}
\usepackage{amssymb}
\usepackage{epstopdf}
\usepackage{multicol}
\usepackage{comment}
\usepackage{parskip}
\theoremstyle{definition}
\newtheorem{definition}{Definition}[section]
\newtheorem{theorem}{Theorem}[section]
\newtheorem{lemma}{Lemma}[section]
\newtheorem{conjecture}{Conjecture}[section]
\newtheorem{corollary}{Corollary}[section]
\newtheorem{proposition}{Proposition}[section]

\title{Goldbach's Conjecture and Euler's $\phi$-Function}
\author{Felix Sidokhine}
\date{}                                           

\begin{document}
\maketitle

\begin{abstract}
In this paper we propose an alternative formulation of the binary and ternary Goldbach conjectures as the systems of equations involving the Euler $\phi$-function.
\end{abstract}

\section{Introduction}

Goldbach's conjecture is known as $2n = p + q$ where $n > 1$ is any integer, $p$ and $q$ are primes.  As one of a possible approach to solving (or at least reformulating) the problem is also a statement regarding an equation involving $\phi$-function (Erdos, Moser and etc.). We propose an additional reformulation of the Goldbach conjecture as a system of equations involving $\phi$-function.
 
\section{Arithmetical Functions}

An invariant characteristic any element $a$ belonging to a set of natural numbers $\mathbb{N}$ is a number of the prime divisors contained in its representation. Let $\pi$ denote a set of all prime numbers of $\mathbb{N}$.

\begin{definition}
$\nu_p: \mathbb{N} \to \mathbb{Z}_+ = \{n \in \mathbb{Z} | n \geq 0\}$ where $p$ is a prime acts on any element $a$ belonging to $\mathbb{N}$ as follows: $\nu_p (a)= \alpha \text{ if } p^\alpha \mid \mid a \text{ and } \nu_p (a)=0 \text{ if } p \nmid a$.
\end{definition}

\begin{definition}
$\nu: \mathbb{N} \to \mathbb{Z}_+$, where the mapping $\nu$ acts on any element $a$ belonging to $\mathbb{N}$ as follows:  $\nu(a)=\sum_{p\in\pi} \nu_p(a)$. The sum $\sum_{p\in\pi} \nu_p(a)$ consists of a finite number of terms. 
\end{definition}

\begin{definition}
$\phi: \mathbb{N} \to \mathbb{N}$, where the mapping $\phi$ acts on any element $a$ belonging to $\mathbb{N}$ as follows:  $\phi(a)$ equals the number of integers among $1,2,…,a$ which are prime to $a$, \cite{Chandrasekharan:1968aa}.

Let us notice some properties of these mappings: $\nu(a) = 0 \iff a = 1$ ;  $\nu(a) = 1 \iff a \text{ is prime}$; $\phi(a) = a - 1\iff \nu(a) = 1$. 
\end{definition}

\section{Prime Conjectures and  Euler's Totient Function}
                                                                                                                                                                                                                                                                                                                                                                                                                                                                                                                                                                                                                                                                                                                                                                                                                                                                                                                                                                                                                                                                                                                                                                                                                                                                                                                                                                                                                                                                                                                                                                                                                                                                                                                                                                                                                                                                                                                                                                                                                                                                                                                                                                                                                                                                                                                                                                                                                                                                                                                                                                                                                                                                                                                                                                                                                                                                                                                                                                                                                                                                                                                                                                                                                                                                                                                                                                                                                                                                                                                                                                                                                                                                                                                                                                                                                                                                                                                                                                                                                                                                                                                                                                                                                                                                                                                                                                                                                                                                                                                                                                                                                                                                                                                                                                                                                                                                                                                                                                                                                                                                                                                                                                                                                                                                                                                                                                                                                                                                                                                                                                                                                                                                               
\subsection{The Bertrand Postulate}
\hfill\\\\
Bertrand's postulate: Between any two integers $n$ and $2n-2$, $n > 3$, there is a prime.

For Bertrand's postulate it is possible the following alternative form:

\begin{lemma}
Between any two integers $n$ and $2n - 2$, $n > 3$, there is a prime if and only if there exists such an integer $x$ satisfying an inequality $0<x<n-2$ that $\nu(n + x) = 1$.
\end{lemma}

Due to lemma 3.1, Bertrand's postulate is equivalent the following statement, namely: 

\begin{theorem}
The equation $\phi(n+x)+1= n+x$ has an integral solution $x$, satisfying an inequality $0<x<n-2$, for every integer $n> 3$ if and only if a system of the Fermat congruences $(n+x)^{\phi(p_i)} \equiv 1 \mod p_i$, where $p_i$ runs a set of all primes $p \leq [\sqrt{n+x}]< \sqrt{n+x} \leq \sqrt{2n-3}$ has a solution $x$, $0<x<n-2$, for every integer $n> 3$.
\end{theorem}

\begin{proof}
Let $x_0$, $0<x_0<n-2$, be a solution of the equation $\phi(n+x)+1= n+x$ hence $n + x_0$ is a prime. Therefore $x_0$, $0<x_0 <n-2$, is a solution of the system of Fermat's consequences $(n+x)^{\phi(p_i)} \equiv 1 \mod p_i$, where $p_i$ runs a set of all primes $p \leq [\sqrt{n+x_0 }]<\sqrt{n+x_0}$. Now let $x_0$, $0<x_0 <n-2$, is a solution of a system of consequences $(n+x)^{\phi(p_i)} \equiv 1 \mod p_i$ where $p_i$ runs a set of all primes $p \leq [\sqrt{n+x_0 }]<\sqrt{n+x_0}$ then, according to trial division, $n + x_0$ is a prime and $x_0$, $0<x_0<n-2$, be a solution of the equation $\phi(n+x)+1= n+x$.
\end{proof}

Due to Tchebycheff's theorem, \cite{Sierpinski:1988aa}, and theorem 3.1 the following statement is true:

\begin{theorem}
A system of the Fermat congruences $(n+x)^{\phi(p_i)} \equiv 1 \mod p_i$, where $p_i$ runs a set of all primes $p \leq [\sqrt{n+x}]<\sqrt{n+x}\leq \sqrt{2n-3}$ has an integral solution $x$, $0<x<n-2$, for every integer $n>3$. A number of solutions of this system of congruences is equal to $\pi(2n-2)-\pi(n)$.
\end{theorem}

An expression $\pi(2n-2) - \pi(n)$ has gotten purely algebraic interpretation as a number of solutions of a system of Fermat's congruences $(n+x)^{\phi(p_i)} \equiv 1 \mod p_i$ where $x$ satisfies an inequality $0<x<n-2$ and a number of the congruences depends on variable $x$ and varies between $\pi(\sqrt{n})$ and $\pi(\sqrt{2n-3})$.

\subsection{The Goldbach Conjecture}

\subsubsection{The Binary Goldbach Conjecture}
\hfill\\\\
The binary Goldbach conjectures: Every even integer $2n$, $n>1$, is a sum of two primes.

For the binary Goldbach conjecture it is possible the following alternative form:

\begin{lemma}
Every even integer $2n$, $n>1$, is a sum two primes if and only if there is such an integer $x$ with a condition $2n+1<x<4n-1$ that $\nu((x-2n)(4n-x))=2$.
\end{lemma}

\begin{proof}
Indeed, let $x$ satisfy an inequality $2n+1<x<4n-1$ and $\nu((x-2n)(4n-x))=2$. Since $x-2n > 1$, $4n - x > 1$ are true hence $x - 2n, 4n - x$ are primes and $2n$ is a sum of two primes. Now let $x - 2n, 4n - x$ be primes so $2n+1<x<4n-1$. It is important to note that if $2n = p + q$ where $p, q$ are primes then there is such an integer $x$ that $p, q$ can be expressed by $x-2n, 4n-x$.
\end{proof}

\begin{theorem}
Goldbach's conjecture is true iff an equation $\nu((x-2n)(4n-x))=2$ with a condition $2n+1<x<4n-1$ has a solution for every integer $n > 1$.
\end{theorem}

\begin{corollary}
The binary Goldbach conjecture is true iff an equation $\nu(n^2-x^2 )=2$  has a solution for every integer $n>2$ with a condition $0\leq x \leq n-3$.
\end{corollary}

\begin{proof}
According to theorem 3.3, $\nu((y-2n)(4n-y))=2$ where $2n+1<y<4n-1$. Making a substitution $y = 3n - x$ we will get $\nu((y-2n)(4n-y))=\nu(n^2-x^2 )=2$ where $0 \leq x \leq n - 3$.
\end{proof}

Using corollary 3.1 we can represent the binary Goldbach conjecture as follows:

\begin{conjecture}
A system of equations involving $\phi$-function
\begin{equation}
\begin{cases}
\phi(n-x)+1=n-x \\
\phi(n+x)+1=n+x
\end{cases}
\end{equation} 
has an integral solution for every natural number $n$ distinct from unity.
\end{conjecture}

The following statement links Goldbach's conjecture with a system of Fermat's congruences:

\begin{theorem}
A system of equations involving $\phi$-function
\begin{equation}
\begin{cases}
\phi(n-x)+1=n-x \\
\phi(n+x)+1=n+x 
\end{cases}
\end{equation}

has an integral solution for every natural number $n$ great than three iff for every $n$ exists such an integer $x$, $0  \leq x  \leq n - 3$, that $x$ is a solution of a system of the Fermat congruences

\begin{equation}
\begin{cases}
(n-x)^{\phi(p_i)} \equiv 1 \mod p_i \\ 
(n+x)^{\phi(q_j)} \equiv 1 \mod q_j 
\end{cases}
\end{equation}
where $p_i, q_j$ run all primes $p \leq[\sqrt{n-x}]<\sqrt{n-x}\leq\sqrt{n}, q \leq [\sqrt{n+x}]<\sqrt{n+x} \leq \sqrt{2n-3}$.
\end{theorem}

\begin{proof}
Indeed, let a system of equations have a solution for any natural number $n$ great than three. According to corollary 3.1 for given $n$ there is $x_0$, $0 \leq x_0 \leq n - 3$ and let $x_0 < n - 3$ as $n - x_0, n + x_0$ are primes so $x_0$ is a solution for a system of the Fermat congruences for all primes $p_i$ and $q_j$ satisfying the inequalities $p_i \leq [\sqrt{n-x_0}]<\sqrt{n-x_0}, q_j \leq [\sqrt{n+x_0}]<\sqrt{n+x_0}$ and so on. Now let for any given $n$ there is such an integer $x_0$, $0\leq x_0 \leq n - 3$ and let $x_0 < n - 3$ is a solution of a system of Fermat's congruences $(n-x)^{\phi(p_i)} \equiv 1 \mod p_i, (n+x)^{\phi(q_j)} \equiv 1 \mod q_j$ for all primes $p_i$ and $q_j$ satisfying the inequalities $p_i \leq [\sqrt{n-x_0}]< \sqrt{n-x_0}, q_j \leq [\sqrt{n+x_0}]< \sqrt{n+x_0}$ then according to trial division $n-x_0,n+x_0$ are primes so a system of equations has a solution and so on.
\end{proof}

\subsubsection{The Ternary Goldbach Theorem}
\hfill\\\\
The ternary Goldbach theorem: Every odd integer $n, n > 5$, is a sum of three primes. 

For the ternary Goldbach theorem it is possible the following alternative form:

\begin{lemma}
Every odd integer $n, n > 5$, is a sum three prime numbers if and only if there are such the integers $x, y$ where $0 \leq y<x<x+y+2<n+1<2x$ that $\nu((n-x-y)(2x-n)(n-x+y)) = 3$.
\end{lemma}

\begin{proof} 
Indeed, let an integer $n, n > 5$, be given and there are such $x, y$ that  $0 \leq y<x<x+y+2<n+1<2x$ and  $\nu((n-x-y)(2x-n)(n-x+y))=3$. Then $n - x - y \geq 2$ and $2x - n \geq 2$ hence that $n - x - y, 2x - n, n - x + y$ are primes and $n$ is a sum of three primes. Now let $n - x - y = p$, $2x - n = q$, $n - x + y = r$  where  $p, q, r$ are primes. As $n = x + y + p$ so $n > x + y + 1$, as $2x = n + q$ so $2x>n+1$ in addition $x > y$. Thus we have $0 \leq y<x<x+y+2<n+1<2x$. It is important to note that if $n = p + q + r$ where $p, q$ and $r$ are primes then there are integers $x, y$ such that $p, q$ and $r$ can be expressed by $n-x-y, 2x-n$ and $n-x+y$.
\end{proof}

Using lemma 3.3 we can represent the ternary Goldbach therem as follows:

\begin{theorem}
A system of equations involving $\phi$-function
\begin{equation}
\begin{cases}
      \phi(n-x-y)+1=n-x-y \\
       \phi(2x-n)+1=2x-n \\
 \phi(n-x+y)+1=n-x+y 
 \end{cases}
 \end{equation}
has a solution $(x, y)$ in non - negative integers for every odd integer $n$ greater than five.
\end{theorem}

\begin{proof}
Theorem 3.5 is a direct consequence of the work \cite{Helfgott:aa}.
\end{proof}

Consider a peculiar case of the ternary Goldbach theorem, namely:

\begin{conjecture}
 A system of equations involving $\phi$ - function is complemented with a congruence 
 \begin{equation}
 \begin{cases}
 \phi(n-x-y)+1=n-x-y \\
 \phi(2x-n)+1=2x-n \\
 \phi(n-x+y)+1=n-x+y \\
 (n-x-y)(2x-n) \equiv 0 \mod 3
\end{cases}
\end{equation}
has a solution $(x, y)$ in non - negative integers for every odd integer $n$ greater than five.
\end{conjecture}

The next statement links a peculiar case of one with the binary Goldbach conjecture: 

\begin{proposition}
Every odd integer greater than five can be represented as a sum of three primes so that at least one of them is a prime number 3 if and only if every even integer greater than two can represent as a sum of two primes.
\end{proposition}

\begin{proof}
The proof is inductive and quite obvious.
\end{proof}

\begin{theorem}
 The peculiar case of the ternary Goldbach theorem is true if and only if the binary Goldbach conjecture is true.
\end{theorem}

\begin{proof}
Theorem 3.6 is a direct consequence of proposition 3.1.
\end{proof}

``Almost all'' odd integers great than five can be written as the sum of three primes so that at least one of them is a prime number 3 (in the sense that the fraction of odd numbers which can be so written tends towards 1). That it is quite obvious due to the works: van der Corput, 1937; Estermann, 1938; Tchudakoff, 1938 and etc. In the case of odd integers it is known every odd integer greater than five can be represented as a sum of three primes. Helfgott, 2014.
 
However a possibility of a presentation of the binary Goldbach conjecture as a peculiar case of the ternary Goldbach theorem puts a number of the questions before circle and sieve methods as in these methods of solving of Goldbach's conjectures the prime numbers are deprived of personal character. Here we can note a quality jump under passing from the ternary Goldbach theorem to the binary Goldbach conjecture which scarcely can overcome by estimating sums over primes.

\section{Conclusion}

We have explored the possibility of reformulating certain problems about primes as the existence of integral solutions of the systems of equations involving the Euler $\phi$ - function.

We have been successful at doing so for the ternary and binary Goldbach conjectures (with binary being an amended system of the ternary).

While solving the systems of equations involving the Euler $\phi$ - function is far from evident, some success has been achieved by studying Fermat's congruences with varying boundary conditions. Nonetheless, it offers alternative path to existing methods, which may be worth exploring.

\bibliography{references} 
\bibliographystyle{IEEEtran}

\end{document}